\documentclass[11pt]{article}

\usepackage{amsmath,amsthm,verbatim,amssymb,amsfonts,amscd, graphicx}
\usepackage{graphics}
\usepackage[bookmarks]{hyperref}
\hypersetup{
    colorlinks=true,   	
    linkcolor=red,      
    citecolor = [rgb]{0 0.7 0},   	
    filecolor=magenta, 	
    urlcolor=blue
}
\usepackage{enumitem}
\usepackage{xcolor}
\usepackage[baseline]{euflag}
\usepackage{appendix}
\newtheorem{theorem}{Theorem}

\newtheorem{lemma}[theorem]{Lemma}
\newtheorem{proposition}[theorem]{Proposition}

\newtheorem{remark}[theorem]{Remark}
\theoremstyle{definition}

\renewcommand{\ref}[1]{(\ref{#1})}

\newcommand{\p}[1]{
{#1}^{\prime}
}

\newcommand{\pmax} {
p_{{\mathrm{max}}}
}

\newcommand{\Var} {
\mathrm{Var}
}

\newcommand{\Un}{
U^{(n)}
}

\newcommand {\nmax} {
N_{\rm max}
}

\newcommand{\ns} {
n\sigma^2
}

\newcommand{\ps}{
p_{S_n}
}
\newcommand{\nsceil}{
\lceil \ns \rceil
}

\newcommand{\floor}[1]{
\lfloor #1 \rfloor
}
\begin{document}

\title{Approximate Discrete Entropy Monotonicity for Log-Concave Sums}
\author{Lampros Gavalakis
\thanks { Laboratoire d’Analyse et de Math\'ematiques Appliqu\'ees, Universit\'e Gustave Eiffel, France.  Email:\texttt{\href{mailto:lampros.gavalakis@univ-eiffel.fr}
			{lampros.gavalakis@univ-eiffel.fr}}.
L.G. has received funding from the European Union's Horizon 2020 research and innovation program
 under the Marie Sklodowska-Curie grant agreement No 101034255. {\large \euflag}
}
}

\maketitle

\begin{abstract}
It is proven that a conjecture of Tao (2010) holds true for log-concave random variables on the integers: 
For every $n \geq 1$, if $X_1,\ldots,X_n$ are i.i.d. integer-valued, log-concave random variables, then
$$
H(X_1+\cdots+X_{n+1}) \geq  H(X_1+\cdots+X_{n}) + \frac{1}{2}\log{\Bigl(\frac{n+1}{n}\Bigr)} - o(1)
$$ 
as $H(X_1) \to \infty$, where $H$ denotes the (discrete) Shannon entropy.
The problem is reduced to the continuous setting by showing that if $U_1,\ldots,U_n$ are independent continuous uniforms on $(0,1)$, then 
$$
h(X_1+\cdots+X_n + U_1+\cdots+U_n) = H(X_1+\cdots+X_n) + o(1)
$$
as $H(X_1) \to \infty$, where $h$ stands for the differential entropy. Explicit bounds for the $o(1)$-terms are provided. 
\end{abstract}


\medskip


\section{Introduction}

\subsection{Monotonic Increase of Differential Entropy}

Let $X, Y$ be two independent random variables with densities in $\mathbb{R}$. The differential entropy of $X$, having density $f$, is
$$
h(X) = -\int_{\mathbb{R}}{f(x)\log{f(x)}dx}
$$
and similarly for $Y$. Throughout $`\log$' denotes the natural logarithm. 

The Entropy Power Inequality (EPI) plays a central role in information theory. It goes back to Shannon \cite{shannon1948mathematical} and was first proven in full generality by Stam \cite{stam}.
It asserts that

\begin{equation} \label{EPI}
N(X+Y) \geq N(X) + N(Y),
\end{equation}
where $N(X)$ is the \textit{entropy power} of $X$:
$$
N(X) = \frac{1}{2 \pi e}e^{2h(X)}.
$$

\noindent
If $X_1,X_2$ are identically distributed, \eqref{EPI} can be rewritten as 
\begin{equation} \label{EPIlogarithmic}
h(X_1+X_2) \geq h(X_1) + \frac{1}{2}\log{2}.
\end{equation}

The EPI is also connected with and has applications in probability theory. The following generalisation is due to Artstein, Ball, Barthe and Naor \cite{artstein}: If $\{X_i\}_{i=1}^{n+1}$ are continuous, i.i.d. random variables then 
\begin{equation} \label{abbn}
h\Bigl(\frac{1}{\sqrt{n+1}}\sum_{i=1}^{n+1}{X_i}\Bigr) \geq h\Bigl(\frac{1}{\sqrt{n}}\sum_{i=1}^{n}{X_i}\Bigr).
\end{equation}
This is the monotonic increase of entropy along the central limit theorem \cite{barronentropy}. The main result of this paper may be seen as an approximate, discrete analogue 
of \eqref{abbn}.

\subsection{Sumset theory for Entropy}

There has been interest in formulating discrete analogues of the EPI from various perspectives \cite{adaptivesensing, abbe, harremoesepibinomial, woo2015discrete}. It is not hard to see that the exact statement \eqref{EPIlogarithmic} can not hold for all discrete random variables by considering deterministic (or even close to deterministic) random variables.

Suppose $G$ is an additive abelian group and $X$ is a random variable supported on a discrete (finite or countable) subset $A$ of $G$ with probability mass function (p.m.f.) $p$ on $G$. 
The Shannon entropy, or simply {entropy} of $X$ is  
\begin{equation} \label{entropynatsequation}
H(X) = -\sum_{x \in A}{p(x)\log{p(x)}}.
\end{equation}
Tao \cite{tao_sumset_entropy} proved that if $G$ is torsion-free and $X$ takes finitely many values then
\begin{equation} \label{taofirststatement}
H(X_1+X_2) \geq H(X_1) + \frac{1}{2}\log{2} - o(1),
\end{equation}
where $X_1, X_2$ are independent copies of $X$ and the $o(1)$-term vanishes as the entropy of $X$ tends to infinity. 
That work explores the connection between additive combinatorics and entropy, which was identified by Tao and Vu in the unpublished notes \cite{taovunotes} and by Ruzsa \cite{ruzsa2009sumsets}. 
The main idea is that random variables in $G$ may be associated with subsets $A$ of $G$: By the \textit{asymptotic equipartition property} \cite{cover}, there is a set $A_n$ (the typical set) such that if $X_1,\ldots,X_n$ are i.i.d. copies of $X$, then $(X_1,\ldots,X_n)$ is approximately uniformly distributed on $A_n$ and $|A_n| = e^{n(H(X) + o(1))}.$ Hence, given an inequality involving cardinalities of sumsets, it is natural to guess that a counterpart statement holds true for random variables if the logarithm of the cardinality is replaced by the entropy. 

Exploring this connection, Tao \cite{tao_sumset_entropy} proved an inverse theorem for entropy, which characterises random variables for which the addition of an independent copy does not increase the entropy by much. This is the entropic analogue of the inverse Freiman theorem \cite{tao_vu_book} from additive combinatorics, which characterises sets for which the sumset is not much bigger than the set itself. The discrete EPI \eqref{taofirststatement} is a consequence of the inverse theorem for entropy. 

Furthermore, it was conjectured in \cite{tao_sumset_entropy} that for any $n \geq 2$ and $\epsilon > 0$ 
\begin{equation} \label{taosoriginalconj}
 H(X_1+\cdots+X_{n+1}) \geq  H(X_1+\cdots+X_{n}) + \frac{1}{2}\log{\Bigl(\frac{n+1}{n}\Bigr)} - \epsilon,
\end{equation}
provided that $H(X)$ is large enough depending on $n$ and $\epsilon$, where $\{X_i\}_{i=1}^{n+1}$ are i.i.d. copies of $X$.
 
We will prove that the conjecture \eqref{taosoriginalconj} holds true for log-concave random variables on the integers. 
An important step in the proof of \eqref{taofirststatement} is reduction to the continuous setting by approximation of the continuous density with a discrete p.m.f.; we briefly outline these key points from that proof in Section \ref{taosection} below as we are going to take a similar approach.

A discrete entropic central limit theorem was recently established in \cite{gavalakis2021entropy}. A discussion relating the above conjecture to the convergence of Shannon entropy to its maximum in analogy with \eqref{abbn} may be found there. 

It has also been of interest to establish finite bounds for the $o(1)$-term \cite{adaptivesensing, woo2015discrete}. Our proofs yield explicit rates for the $o(1)$-terms, which are exponential in $H(X_1)$. 

The class of discrete log-concave distributions has been considered recently by Bobkov, Marsiglietti and Melbourne \cite{bobkov2022concentration} in connection with the EPI. In particular, discrete analogues of \eqref{EPI} were proved for this class.
In addition, sharp upper and lower bounds on the maximum probability of discrete log-concave random variables in terms of their variance were provided, which we are going to use in the proofs (see Lemma \ref{bobkovlemma} below). Although log-concavity is a strong assumption in that it implies, for example, connected support set and moments of all orders, many important distributions are log-concave, e.g. Bernoulli, Poisson, geometric, negative binomial and others. 

\subsection{Main results and proof ideas} \label{taosection}

The first step in the proof method of \cite[Theorem 1.9]{tao_sumset_entropy} is to assume that $H(X_1+X_2) \leq H(X_1) + \frac{1}{2}\log{2} - \epsilon.$ Then, because of \cite[Theorem 1.8]{tao_sumset_entropy}, proving the result for random variables $X$ that can be expressed as a sum $Z+U$, where $Z$ is a random variable with entropy $O(1)$ and $U$ is a uniform on a large arithmetic progression, say $P$, suffices to get a contradiction. Such random variables satisfy, for every $x$,
\begin{equation} \label{taoprobs}
\mathbb{P}(X = x) \leq \frac{C}{|P|}
\end{equation} 
for some absolute constant $C$. 
Using tools from the theory of sum sets, it is shown that it suffices to consider random variables that take values in a finite subset of the integers. For such random variables that satisfy \eqref{taoprobs}, the smoothness property
\begin{equation} \label{tvproperty}
\|p_{X_1+X_2} - p_{X_1+X_2+1}\|_{\rm TV} \to 0
\end{equation}
as $H(X) \to \infty$ is established, where $p_{X_1+X_2}, p_{X_1+X_2+1}$ are the p.m.f.s of $X_1+X_2$ and $X_1+X_2+1$ respectively and $\|\cdot\|_{\mathrm{TV}}$ is the total variation distance defined in \eqref{tvdef} below. 
Using this, it is shown that 
\begin{equation} \label{introentropyapprox}
h(X_1+X_2+U_1+U_2) = H(X_1+X_2) + o(1),
\end{equation}
as $H(X) \to \infty,$ where $U_1,U_2$ are independent continuous uniforms on $(0,1).$  The EPI for continuous random variables is then invoked.

The tools that we use are rather probabilistic-- our proofs lack any additive combinatorial arguments as we already work with random variables on the integers, which have connected support set. An important technical step in our case is to show that any log-concave random variable $X$ on the integers satisfies
$$
\|p_X- p_{X+1}\|_{\rm TV} \to 0
$$
as $H(X) \to \infty$. Using this we show a generalisation of \eqref{introentropyapprox}, our main technical tool: 

\begin{theorem} \label{maintheorem}
Let $n \geq 1$ and suppose $X_1,\ldots,X_n$ are i.i.d. log-concave random variables on the integers with common variance $\sigma^2.$ Let $U_1,\ldots,U_n$ be continuous i.i.d. uniforms on $(0,1).$ Then 
\begin{equation} 
h(X_1+\cdots+X_n+U_1+\cdots+U_n) = H(X_1+\cdots+X_n) + o(1),
\end{equation}
where the $o(1)$-term vanishes as $\sigma^2 \to \infty$ depending on $n$. In fact, this term can be bounded absolutely by 
\begin{equation} \label{nsigmarate}
2^{n+6} e^{-(\sqrt{n}\sigma)^{1/5}} (\sqrt{n}\sigma)^{3} + \frac{2^{n+2}}{\sigma\sqrt{n}}\log(2^{n+2}\sigma\sqrt{n}) +\frac{\log{(\ns)}}{8\ns},
\end{equation}
provided that $\sigma > \max\{2^{n+2}/\sqrt{n},3^7/\sqrt{n}\}$.
\end{theorem}
\begin{remark}
{\ Note} that always, $H(X) \to \infty$ implies $\sigma^2 \to \infty,$ since by the maximum entropy property of the Gaussian distribution~\cite{cover} 
\begin{equation}
H(X) = h(X+U) \leq \frac{1}{2}\log{\bigl(2\pi e\sigma^2(1+\frac{1}{12})\bigr)},
\end{equation}
where $U$ is an independent continuous uniform on $(0,1)$. Conversely, for the class of log-concave random variables $\sigma \to \infty$ implies $H(X) \to \infty$, e.g. by Proposition \ref{qlogconcave} and Lemma \ref{bobkovlemma}. Indeed these give a quantitative comparison between $H(X)$ and $\sigma^2=\mathrm{Var}(X)$ for log-concave random variables: $H(X) \geq \log{\sigma}$ for $\sigma \geq 1$ .
\end{remark}
Our main tools are first, to approximate the density of the log-concave sum convolved with the sum of $n$ continuous uniforms with the discrete p.m.f. (Lemma \ref{densityapproxlemma}) and second, to show a type of concentration for the ``information density'', $-\log{p(S_n)},$ using Lemma \ref{nzerolemma}. It is a standard argument to show that log-concave p.m.f.s have exponential tails, since the sum of the probabilities is convergent. Lemma \ref{nzerolemma} is a slight improvement in that it provides a bound for the ratio depending on the variance. 

By an application of the generalised EPI for continuous random variables, we show that the conjecture \eqref{taosoriginalconj} is true for log-concave random variables on the integers, with an explicit dependence between $H(X)$ and $\epsilon.$ Our main result is: 

\begin{theorem} \label{taosconjecturelog}
Let $n \geq 1$ and $\epsilon \in (0,1)$. Suppose $X_1,\ldots,X_n$ are i.i.d. log-concave random variables on the integers. Then if 
$H(X_1)$ is sufficiently large depending on $n$ and $\epsilon$, 
\begin{equation} \label{mainequation}
 H(X_1+\cdots+X_{n+1}) \geq  H(X_1+\cdots+X_{n}) + \frac{1}{2}\log{\Bigl(\frac{n+1}{n}\Bigr)} - \epsilon.
\end{equation}
In fact, for \eqref{mainequation} to hold it suffices to take $H(X_1) \geq \log{\frac{2}{\epsilon}} + \log{\log{\frac{2}{\epsilon}}} + n + 27.$  
\end{theorem}
The proofs of Theorems \ref{maintheorem} and \ref{taosconjecturelog} are given in Section \ref{proofsection}. Before that, in Section \ref{prelimsection} below, we prove some preliminary facts about discrete, log-concave random variables. 

For $n=1,$ the lower bound for $H(X_1)$ given by Theorem \ref{taosconjecturelog} for the case of log-concave random variables on the integers is a significant improvement on the lower bound that can be obtained from the 
proof given in \cite{tao_sumset_entropy} for discrete random variables in a torsion free group, which is $\Omega\Bigl({\frac{1}{\epsilon}}^{{\frac{1}{\epsilon}}^{\frac{1}{\epsilon}}}\Bigr)$.

Finally, let us note that Theorem \ref{maintheorem} is a strong result: Although we suspect that the assumption of log-concavity may be relaxed, we do not expect it to hold in much greater generality; we believe that some structural conditions on the random variables should be necessary.

\section{Notation and Preliminaries} \label{prelimsection}
For a random variable $X$ with p.m.f. $p$ on the integers denote 
 \begin{equation} \label{qdef}
q := \sum_{k \in \mathbb{Z}}\min\{p(k),p(k+1)\}.
\end{equation}
The parameter $q$ defined above plays an important role in a technique known as Bernoulli part decomposition, which has been used in \cite{daviselementary, mcdonald,mineka} to prove local limit theorems. It was also used in \cite{gavalakis2021entropy} to prove the discrete entropic CLT mentioned in the Introduction. 

In the present article, we use $1-q$ as a measure of smoothness of a p.m.f. on the integers. 
In what follows we will also write $q(p)$ to emphasise the dependence on the p.m.f. $p$.

For two p.m.f.s on the integers $p_1$ and $p_2$, we use the notation 
$$\|p_1 - p_2\|_{1} := \sum_{k \in \mathbb{Z}}{|p_1(k)-p_2(k)|}$$
 for the $\ell_1$-distance between $p_1$ and $p_2$ and 
\begin{equation} \label{tvdef}
\|p_1 - p_2\|_{\rm TV} := \frac{1}{2}\|p_1 - p_2\|_{1}
\end{equation}
 for the total variation distance.

\begin{proposition} \label{qTVexpr}
Suppose $X$ has p.m.f. $p_X$ on $\mathbb{Z}$ and let $q = \sum_{k \in \mathbb{Z}}{\min{\{p_X(k),p_X(k+1)\}}}$. Then 
$$\|p_X - p_{X+1}\|_{\rm TV} = 1-q.$$
\end{proposition}
\begin{proof}
Since $|a-b| = a + b - 2\min{\{a,b\}}$,
$$
\|p_X-p_{X+1}\|_1 = \sum_{k\in \mathbb{Z}}{|p_X(k) - p_X(k+1)|} = 2 - 2\sum_{k \in \mathbb{Z}}{\min{\{p_X(k),p_X(k+1)\}}} = 2(1-q).
$$
The result follows. 
\end{proof}

A p.m.f. $p$ on $\mathbb{Z}$ is called {\em log-concave}, if for any $k \in \mathbb{Z}$ 
\begin{equation} \label{logconcavity}
p(k)^2 \geq p(k-1)p(k+1).
\end{equation}
If a random variable $X$ is distributed according to a log-concave p.m.f. we say that $X$ is log-concave. 
Throughout we suppose that $X_1,\ldots,X_n$ are i.i.d. random variables having a log-concave p.m.f. on the integers, $p$, common variance $\sigma^2$ and denote their sum with $S_n$.
Also, we denote
$$
\pmax = \pmax(X) := \sup_k{\mathbb{P}(X = k)}
$$
and write 
\begin{equation} \label{Nmaxdef}
\nmax = \nmax(X) := \max\{k \in \mathbb{Z}: p(k) = \pmax\},
\end{equation}
i.e. $\nmax$ is the last $k \in \mathbb{Z}$ for which the maximum probability is achieved. 
We will make use of the following bound from \cite{bobkov2022concentration}:
\begin{lemma} \label{bobkovlemma}
Suppose $X$ has discrete log-concave distribution with $\sigma^2 = \Var{(X)} \geq 1.$ Then 
\begin{equation} \label{bobkovbound}
\frac{1}{4\sigma} \leq \pmax \leq \frac{1}{\sigma}.
\end{equation}
\end{lemma}
\begin{proof}
Follows immediately from \cite[Theorem 1.1.]{bobkov2022concentration}.
\end{proof}
\begin{proposition} \label{nmaxprop}
Let $X$ be a log-concave random variable on the integers with mean $\mu \in \mathbb{R}$ and variance $\sigma^2$, and let $\delta >0$. Then,
if $\sigma > 4^{1/2\delta},$ 
\begin{equation} \label{nmaxsigma2bound}
 |\nmax - \mu| < \sigma^{3/2+\delta} +1.
\end{equation}
\end{proposition}
\begin{proof}
Suppose for contradiction that $ |\nmax - \mu| \geq \sigma^{3/2+\delta}+1$. Then, using \eqref{bobkovbound},
$$
\mathbb{P}(|X - \mu| > \sigma^{3/2+\delta}) \geq p(\nmax) \geq \frac{1}{4\sigma}.
$$
But Chebyshev's inequality implies
$$
\mathbb{P}(|X - \mu| > \sigma^{3/2+\delta}) \leq \frac{1}{\sigma^{1 + 2\delta}} < \frac{1}{4\sigma}.
$$
\end{proof}

Below we show that for any integer-valued random variable $X$, $q \to 1$ implies $H(X) \to \infty$. 
It is not hard to see that the converse is not always true, i.e. $H(X) \to \infty$ does not necessarily imply $q \to 1$: Consider a random variable with a mass of $\frac{1}{2}$ at zero and all other probabilities equal on an increasingly large subset of $\mathbb{Z}$.
Nevertheless, using Lemma \ref{bobkovlemma}, we show that if $X$ is {log-concave} this implication is true. In fact, part \ref{1minusqupper} of Proposition \ref{qlogconcave} holds for all unimodal distributions. Clearly any log-concave distribution is unimodal, since  \eqref{logconcavity} is equivalent to the sequence $\{\frac{p_X(k+1)}{p_X(k)}\}_{k \in \mathbb{Z}}$ being non-increasing.

\noindent

\begin{proposition} \label{qlogconcave}
Suppose that the random variable $X$ has p.m.f. $p_X$ on the integers and let $q = q(p_X)$ as above. Then

\begin{enumerate} [label = (\roman*)] 
\item \label{1minusqlower} $e^{-H(X)} \leq 1-q.$
\item \label{1minusqupper} If $p_X$ is unimodal, then 
 $1 - q = \pmax. $
\end{enumerate}
\end{proposition}

\begin{proof}
Let $m$ be a mode of $X$, that is $p(m) = \pmax$. Then 
\begin{align*}
q &=\sum_{k \leq  m - 1}{\min{\{p_X(k),p_X(k+1)\}}}  + \sum_{k \geq m}{\min{\{p_X(k),p_X(k+1)\}}} \\
&\leq \sum_{k \leq m-1}{p_X(k)}  + \sum_{k \geq m}{p_X(k+1)} \\
&= 1 - \pmax.
\end{align*}
The bound \ref{1minusqlower} follows since 
$H(X) = \mathbb{E}\bigl({\log{\frac{1}{p_X(X)}}\bigr)} \geq \log{\frac{1}{\max_k{p_X(k)}}}.$

For \ref{1minusqupper}, note that since $p_X$ is unimodal, ${p_X(k+1)} \geq {p_X(k)}$ for all $k < m$ and ${p_X(k+1)} \leq {p_X(k)}$ for all $k \geq m$.
Therefore, the inequality in part \ref{1minusqlower} is equality
and \ref{1minusqupper} follows.
\end{proof}

\section{Proofs of Theorems \ref{maintheorem} and \ref{taosconjecturelog}} \label{proofsection}

Let $U^{(n)} := \sum_{i=1}^n{U_i},$ where $U_i$ are i.i.d. continuous uniforms on $(0,1).$ Let $f_{S_n+\Un}$ denote the density of $S_n+\Un.$ We approximate $f_{S_n+U^{(n)}}$ with the p.m.f., say $p_{S_n}$, of $S_n$.

\noindent
We recall that the class of discrete log-concave distributions is closed under convolution \cite{hoggar} and hence the following lemma may be applied to $S_n$.
\begin{lemma} \label{densityapproxlemma}
Let $S$ be a log-concave random variable on the integers with variance $\sigma^2 = \Var{(S)}$ and, for any $n \geq 1,$ denote by $f_{S+U^{(n)}}$ the density of $S + U^{(n)}$ on the real line. Then for any $n \geq 1$ and $x \in \mathbb{R}$,
\begin{equation} \label{densityapprox}
f_{S+\Un}(x) = p_{S}(\floor{x}) + g_n(\floor{x},x),
\end{equation}
for some $g_n:\mathbb{Z}\times\mathbb{R} \to \mathbb{R}$ satisfying
\begin{equation} \label{gkbounds}
\sum_{{k} \in \mathbb{Z}}\sup_{u\in [k,k+1)}{|g_n(k,u)|}\leq (2^{n} -2)\frac{1}{\sigma}.
\end{equation}
Moreover, if $\floor{x} \geq \nmax + n - 1,$
\begin{equation} \label{fsnleqpn}
f_{S+\Un}(x) \leq 2^np_{S}(\floor{x}-n+1).
\end{equation}
\end{lemma}

\begin{proof}

First we recall that for a discrete random variable $S$ and a continuous independent random variable $U$ with density $f_U$, $S+U$ is continuous with density 
$$
f_{S+U}(x) = \sum_{k \in \mathbb{Z}}{p_S(k)f_U(x-k)}.
$$

For $n=1$, the statement is true with $g_n = 0$. We proceed by induction on $n$ with $n=2$ as base case, which illustrates the idea better. The density of $U_1+U_2$ is $f_{U_1+U_2}(u) = u,$ for $u\in (0,1)$ and $f_{U_1+U_2}(u) = 2-u $, for $u \in [1,2)$. Thus, we have 

\begin{equation} \label{refforsecondind}
f_{S+U^{(2)}}(x) = p_{S}(\floor{x}) + (1-x+\floor{x})(p_{S}(\floor{x}-1) - p_{S}(\floor{x})).
\end{equation}
Therefore,
$$
f_{S+U^{(2)}}(x) = p_{S}(\floor{x}) + g_2(\floor{x},x),
$$
where $g_2(k,x) = (1-x+{k})(p_{S}(k-1) - p_{S}(k))$ and by Propositions \ref{qTVexpr}, \ref{qlogconcave}.\ref{1minusqupper} and Lemma \ref{bobkovlemma}
$$
\sum_{{k} \in \mathbb{Z}}\sup_{u\in[k,k+1)}|g_2(k,u)| \leq \sum_{k \in \mathbb{Z}}{|p_S(k) - p_S(k-1)|} = \|p_{S} - p_{S+1}\|_{1} \leq \frac{2}{\sigma}.
$$

Next, we have 
\begin{align} \nonumber
f_{S+U^{(n+1)}}(x) &= \int_{(0,1)}{f_{S+U^{(n)}}(x-u)du} \\ \label{twointegrals}
&= \int_{(0,1)\cap(x-\floor{x}-1,x-\floor{x})}{f_{S+U^{(n)}}(x-u)du} + \int_{(0,1)\cap(x-\floor{x},x-\floor{x}+1)}{f_{S+U^{(n)}}(x-u)du}.
\end{align}
Using the inductive hypothesis, \eqref{twointegrals} is equal to 
\begin{align} \nonumber
&(x-\floor{x}) p_S(\floor{x}) + (1-x+\floor{x}) p_S(\floor{x}-1) \\
&+ \int_{(0,1)\cap(x-\floor{x}-1,x-\floor{x})}{g_{n}(\floor{x},x-u)du} + \int_{(0,1)\cap(x-\floor{x},x-\floor{x}+1)}{g_n(\floor{x}-1,x-u)du}
\end{align}
with $g_n$ satisfying \eqref{gkbounds}. Thus, we can write 
\begin{align} \nonumber
f_{S+U^{(n+1)}}(x) & = p_S(\floor{x}) + (1-x+\floor{x})(p_S(\floor{x}-1) - p_S(\floor{x})) \\ 
&+ \int_{(0,1)\cap(x-\floor{x}-1,x-\floor{x})}{g_{n}(\floor{x},x-u)du} + \int_{(0,1)\cap(x-\floor{x},x-\floor{x}+1)}{g_n(\floor{x}-1,x-u)du} \\ 
&= p_S(\floor{x}) +g_{n+1}(\floor{x},x),
\end{align}
\noindent
where $g_{n+1}(k,x) =(1-x+k)(p_S(k-1) - p_S(k))+  \int_{(0,1)\cap(x-k-1,x-k)}{g_{n}(k,x-u)du} $\\$+ \int_{(0,1)\cap(x-k,x-k+1)}{g_n(k-1,x-u)du}$. Therefore, since $g_n$ satisfies \eqref{gkbounds}, 
\begin{equation}
\sum_{{k} \in \mathbb{Z}}\sup_{u\in[k,k+1)}|g_{n+1}({k},u)| \leq  \frac{2}{\sigma} + 2\sum_{k \in \mathbb{Z}}{\sup_{u\in[k,k+1)}|g_{n}({k},u)|} \leq  \frac{2}{\sigma} + 2(2^{n} -2)\frac{1}{\sigma}  = (2^{n+1} - 2)\frac{1}{\sigma},
\end{equation}
completing the inductive step and thus the proof of \eqref{gkbounds}.

Inequality \eqref{fsnleqpn} may be proved in a similar way by induction: For $n=2$, by \eqref{refforsecondind}
\begin{equation}
f_{S+U^{(2)}}(x)  \leq 2p_S(\floor{x}-1), 
\end{equation}
since $p_S(\floor{x}) \leq p_S(\floor{x}-1)$ for $\floor{x} \geq \nmax+1.$

By \eqref{twointegrals} and the inductive hypothesis
\begin{align}
f_{S+U^{(n+1)}}(x) &= \int_{(0,1)\cap(x-\floor{x}-1,x-\floor{x})}{f_{S+U^{(n)}}(x-u)du} + \int_{(0,1)\cap(x-\floor{x},x-\floor{x}+1)}{f_{S+U^{(n)}}(x-u)du} \\
&\leq 2^n p_S(\floor{x} - n) + 2^np_S(\floor{x}-n+1) \\
&\leq 2^{n+1}p_S(\floor{x}-n)
\end{align}
completing the proof of \eqref{fsnleqpn} and thus the proof of the lemma.
\end{proof}


\begin{lemma} \label{nzerolemma}
Let $X$ be a log-concave random variable on the integers with p.m.f. $p$, mean zero and variance $\sigma^2,$ and let  $0<\epsilon < 1/2$. If $\sigma \geq \max{\{3^{1/\epsilon},(12e^3)^{1/(1-2\epsilon)}\}},$ there is an $N_0 \in \{\nmax,\ldots,\nmax+2\lceil\sigma^2\rceil\}$ such that, for each $k \geq N_0,$ 
\begin{equation} \label{exponentiallysmallprobs}
p(k+1) \leq \biggl(1-\frac{1}{\sigma^{2-\epsilon}}\biggr) p(k).
\end{equation}

Similarly, there is an $N_0^- \in \{\nmax-2\lceil\sigma^2\rceil,\ldots,\nmax\}$ such that, for each $k \leq N_0^-,$ 
$$p(k-1) \leq \biggl(1-\frac{1}{\sigma^{2-\epsilon}}\biggr) p(k).$$
\end{lemma}

\begin{proof}
Let $\theta = 1-\frac{1}{\sigma^{2-\epsilon}}.$ It suffices to show that there is an $N_0 \in \{\nmax,\ldots,\nmax+2\lceil\sigma^2\rceil\}$ such that $p(N_0+1) \leq \theta p(N_0),$ since then, for each $k \geq N_0, \frac{p(k+1)}{p(k)} \leq \frac{p(N_0+1)}{p(N_0)} \leq \theta$ by log-concavity.

Suppose for contradiction that $p(k+1) \geq \theta p(k)$ for each $ k \in \{\nmax,\ldots,\nmax+2\lceil\sigma^2\rceil\}$. 
Then, we have, using \eqref{bobkovbound}
\begin{align}
\sigma^2 &= \sum_{k \in \mathbb{Z}}{k^2p(k)} \geq \sum_{k = \nmax}^{\nmax+2\lceil\sigma^2\rceil}{k^2p(k)} \geq \sum_{k = \nmax}^{\nmax+2\lceil\sigma^2\rceil}{k^2\theta^{k-\nmax}\frac{1}{4\sigma}} \\  \label{nmaxmtheta}
&= \sum_{m=0}^{2\lceil\sigma^2\rceil}(\nmax+m)^2\theta^m\frac{1}{4\sigma} \geq \sum_{m=\max\{0,-\nmax\}}^{2\lceil\sigma^2\rceil}(\nmax+m)^2\theta^m\frac{1}{4\sigma}.
\end{align}
Now we use Proposition \ref{nmaxprop} with $\delta >0$ to be chosen later. Thus, the right-hand side of \eqref{nmaxmtheta} is at least
\begin{align}\nonumber
&\sum_{m=\lceil\sigma^{3/2+\delta}+1\rceil}^{2\lceil\sigma^2\rceil}(\nmax+m)^2\theta^m\frac{1}{4\sigma} \\
&\geq \sum_{m=\lceil\sigma^{3/2+\delta}\rceil+1}^{2\lceil\sigma^2\rceil}(m - \lceil\sigma^{3/2+\delta}\rceil-1)^2\theta^m\frac{1}{4\sigma} \\
&= \sum_{k=0}^{2\lceil\sigma^2\rceil - \lceil\sigma^{3/2+\delta}\rceil-1}k^2\theta^{k+\lceil\sigma^{3/2+\delta}\rceil+1}\frac{1}{4\sigma} \\ \label{suboptimal}
&\geq 
\theta^{\lceil\sigma^{3/2+\delta}\rceil+1}\frac{1}{4\sigma}\sum_{k=1}^{\lceil\sigma^2\rceil -1 }k\theta^{k}  \\  \label{thetasexpr} 
&= \frac{\theta^{\lceil\sigma^{3/2+\delta}\rceil+1}}{4\sigma}\Bigl[\theta \frac{1-\theta^{\lceil\sigma^2\rceil}}{(1-\theta)^2} - \lceil\sigma^2\rceil\frac{\theta^{\lceil\sigma^2\rceil}}{(1-\theta)} \Bigr].
\end{align}
Using the elementary bound $(1-x)^y \geq e^{-2xy},$ for $0<x < \frac{\log{2}}{2},y>0$, we see that 
$$
\theta^{\lceil\sigma^{3/2+\delta}\rceil+1} = (1-\frac{1}{\sigma^{2-\epsilon}})^{\lceil\sigma^{3/2+\delta}\rceil+1} \geq e^{-2\frac{\sigma^{3/2+\delta}+2}{\sigma^{2-\epsilon}}} \geq e^{-3},
$$
where the last inequality holds as long as $\epsilon+\delta < 1/2.$ Choosing $\delta = 1/4-\epsilon/2,$ we see that the assumption of Proposition \ref{nmaxprop} is satisfied for $\sigma>16^{1/(1-2\epsilon)}$ and thus for $\sigma>(12e^3)^{1/(1-2\epsilon)}$ as well.
Furthermore, using $(1-x)^y \leq e^{-xy}, 0<x<1, y >0,$ we get $\theta^{\sigma^2} \leq e^{-\sigma^{\epsilon}}.$ Thus, the right-hand side of \eqref{thetasexpr} is at least
\begin{align} \nonumber
&\frac{1}{4e^3\sigma}\Bigl[\Bigl(1-\frac{1}{\sigma^{2-\epsilon}}\Bigr) \bigl(1-e^{-\sigma^{\epsilon}}\bigr)\sigma^{4-2\epsilon} - \sigma^{4-\epsilon}e^{-\sigma^{\epsilon}} - \sigma^{2-\epsilon}e^{-\sigma^{\epsilon}}\Bigr] \\ \nonumber
&\geq \frac{\sigma^{3-2\epsilon}}{4e^3}\Bigl[1-2\frac{\sigma^{\epsilon}}{e^{\sigma^\epsilon}} -\frac{1}{\sigma^{2-\epsilon}}\Bigr] \\ \label{3tothe1overeps}
&> \sigma^2\frac{\sigma^{1-2\epsilon}}{12e^3} \\ \label{geqsigma2}
&\geq \sigma^2,
\end{align}
where \eqref{3tothe1overeps} holds for $\sigma > 3^{\frac{1}{\epsilon}}$, since then $\frac{\sigma^{\epsilon}}{e^{\sigma^{\epsilon}}} \leq \frac{1}{4}$ and $\sigma^{-2+\epsilon} < \sigma^{-1} < \frac{1}{9}$. Finally, \eqref{geqsigma2} holds for $\sigma > (12e^3)^{\frac{1}{1-2\epsilon}},$ getting the desired contradiction.

For the second part, apply the first part to the log-concave random variable $-X$. 
\end{proof}

\begin{remark}
The bound \eqref{exponentiallysmallprobs} may be improved due to the suboptimal step \eqref{suboptimal}, e.g. by means of the identity 
$$
\sum_{k=1}^M{k^2\theta^k} = \theta\frac{{d}}{d\theta}\biggl(\sum_{k=1}^M{k\theta^k}\biggr) = \theta \frac{{d}}{d\theta} \Biggl(\theta\frac{{d}}{d\theta}\biggl(\sum_{k=1}^M{\theta^k}\biggr) \Biggr) = \theta \frac{{d}}{d\theta} \Biggl(\theta\frac{{d}}{d\theta}\biggl(\frac{\theta-\theta^{M+1}}{1-\theta}\biggr) \Biggr).
$$
It is, however, sufficient for our purpose as it will only affect a higher-order term in the proof of Theorem \ref{maintheorem}. 
\end{remark}

We are now ready to give the proof of Theorem \ref{maintheorem} and of our main result, Theorem \ref{taosconjecturelog}. 
\begin{proof}[Proof of Theorem \ref{maintheorem}]
Assume without loss of generality that $X_1$ has zero mean. 
Let $F(x) = x\log{\frac{1}{x}}, x>0$ and note that $F(x)$ is non-decreasing for $x \leq 1/e$. As before denote $S_n = \sum_{i=1}^n{X_i},$  $\Un = \sum_{i=1}^n{U_i}$ and let $f_{S_n+\Un}$ be the density of $S_n+\Un$ on the reals. We have 
\begin{align} \nonumber
&h(X_1+\cdots+X_n+U_1+\cdots+U_n)\\ \label{sumbreak}
&= \sum_{k \in (-5\ns,5\ns)}{\int_{[k,k+1)}{F(f_{S_n+\Un}(x)) dx}} + \sum_{|k| \geq 5\ns}{\int_{[k,k+1)}{F(f_{S_n+\Un}(x)) dx}}.
\end{align}

First we will show that the ``entropy tails", i.e. the second term in \eqref{sumbreak}, vanish as $\sigma^2$ grows large. 
To this end, note that for $k \geq 5\ns,$ we have $p_{S_n}(k+1) \leq p_{S_n}(k),$ since by Proposition \ref{nmaxprop} applied to the log-concave random variable $S_n$, $\nmax \leq \ns+1$ as long as $\sqrt{n}\sigma>4$. 
Thus, by \eqref{fsnleqpn}, for $k\geq 5\ns$ and $x \in[k,k+1),$ $f_{S_n+\Un}(x) \leq 2^n\ps(k-n+1)$.
Hence, for 
\begin{equation} \label{pmaxsmall}
\sigma > \frac{2^{n}}{\sqrt{n}}e,
\end{equation}
we have, using the monotonicity of $F$ for $x\leq \frac{1}{e}$,
\begin{align} \label{monofFand}
0 &\leq \sum_{k \geq 5\ns}{\int_{[k,k+1)}{F(f_{S_n+\Un}(x)) dx}} \leq \sum_{k \geq 5\ns}{F\bigl(2^n p_{S_n}(k-n+1)\bigr) } \\ \label{ctsdiscretetails}
&= \sum_{k \geq 5\ns}{2^np_{S_n}(k-n+1)\log{\frac{1}{2^np_{S_n}(k-n+1)}}}\\ \label{thetalemmasigmacondition}
&\leq 2^n\frac{1}{\sqrt{n}\sigma} \sum_{k \geq 5\ns}{\theta^{k-4\nsceil}\log{\frac{\sqrt{n}\sigma}{2^n\theta^{k-4\nsceil}}}} \\
&=  2^n\frac{\log{\frac{1}{\theta}}}{\sqrt{n}\sigma} \sum_{m \geq \ns}{m\theta^{m}} +  2^n\frac{\log{\frac{\sqrt{n}\sigma}{2^n}}}{\sqrt{n}\sigma} \sum_{m \geq \ns}{\theta^{m}}\\
&\leq 2^{n+1} \frac{\log{\sqrt{n}\sigma}}{\sqrt{n}\sigma}     \Bigl[ \frac{\theta^{\nsceil+1}}{(1-\theta)^2} +\nsceil \frac{\theta^{\nsceil}}{1-\theta} +\frac{\theta^{\nsceil}}{1-\theta}      \Bigr] \\
&\leq 2^{n+1} \frac{\log{\sqrt{n}\sigma}}{\sqrt{n}\sigma}e^{-(\sqrt{n}\sigma)^{\epsilon}}  \Bigl[ (\sqrt{n}\sigma)^{4-2\epsilon} + (n\sigma^2+1)(\sqrt{n}\sigma)^{2-\epsilon} + (\sqrt{n}\sigma)^{2-\epsilon}    \Bigr]  \\ 
&\leq 2^{n+3} \frac{\log{\sqrt{n}\sigma}}{\sqrt{n}\sigma}e^{-(\sqrt{n}\sigma)^{\epsilon}} (\sqrt{n}\sigma)^{4-\epsilon} \\ \label{tailfinalbound}
&\leq 2^{n+4} e^{-(\sqrt{n}\sigma)^{1/5}} (\sqrt{n}\sigma)^{3}.
\end{align}
Here \eqref{thetalemmasigmacondition} holds for 
\begin{equation} \label{sigma3tothe7} 
\sqrt{n}\sigma > 3^7
\end{equation}
 with $\theta = 1-\frac{1}{(\sqrt{n}\sigma)^{2-\epsilon}}=1-\frac{1}{(\sqrt{n}\sigma)^{9/5}}$, where we have used Lemma \ref{nzerolemma} with $\epsilon = 1/5$ (which makes the assumption approximately minimal). In particular, repeated application of \eqref{exponentiallysmallprobs} yields \\
$p_{S_n}(k-n+1) \leq \theta^{k-4\nsceil}p_{S_n}\bigl(4\nsceil-n+1\bigr) \leq \frac{\theta^{k-4\nsceil}}{\sqrt{n}\sigma}.$

We bound the left tail in the exact the same way, using the second part of Lemma \ref{nzerolemma}:
\begin{equation}\label{lefttail}
0 \leq \sum_{k \leq -5\ns}{\int_{[k,k+1)}{F(f_{S_n+\Un}(x)) dx}} \leq 2^{n+4} e^{-(\sqrt{n}\sigma)^{1/5}} (\sqrt{n}\sigma)^{3}.
\end{equation}

Next we will show that the first term in \eqref{sumbreak} is approximately $H(S_n)$ to complete the proof. 
We have 
\begin{align} \nonumber
&\sum_{k \in (-5\ns,5\ns)}{\int_{[k,k+1)}{F(f_{S_n+\Un}(x)) dx}}
= \log{\ns}\int_{(-\lfloor5\ns\rfloor,\lfloor 5\ns \rfloor +1)}{f_{S_n+\Un}(x)dx}  \\ 
&+ \sum_{k \in  (-5\ns,5\ns)}{\int_{[k,k+1)}{F(f_{S_n+\Un}(x)) - f_{S_n+\Un}(x)\log{\ns}dx}} \\ \nonumber
&= \log{\ns}\mathbb{P}\bigl(S_n+\Un \in  (-\lfloor 5\ns \rfloor,\lfloor 5\ns \rfloor +1)\bigr)  \\ \label{minusaddns}
&+ \sum_{k \in  (-5\ns,5\ns)}{\int_{[k,k+1)}{F(f_{S_n+\Un}(x)) - f_{S_n+\Un}(x)\log{\ns}dx}}.
\end{align}
Now we will apply the estimate of Lemma \ref{elementaryestimatelemma}, which is stated and proved in the Appendix, to the integrand of the second term in \eqref{minusaddns} with $G(x) = F(x) - x\log{(n\sigma^2)},$ $\mu = \frac{1}{11\sigma\sqrt{n}}, D = 2^n\sqrt{n}\sigma, M = \ns, a = f_{S_n+\Un}(x)$ and $b=p_{S_n}(k).$
We obtain, using Lemma \ref{densityapproxlemma},
\begin{align} \nonumber
& \sum_{k \in  (-5\ns,5\ns)}{\Bigl|\int_{[k,k+1)}{G(f_{S_n+\Un}(x))dx}} - {G(p_{S_n}(k))} \Bigr| \\ \label{importantterms}
&\leq 11\ns\frac{2\log{(11\sigma\sqrt{n})}}{11\sigma^3n\sqrt{n}} + \sum_{k \in \mathbb{Z}}{\int_{[k,k+1)}|f_{S_n+\Un}(x) - p_{S_n}(k)|dx}\Bigl(\log{(11\sigma\sqrt{n})} + \log(e2^n\sigma\sqrt{n})\Bigr) \\ \label{errorterm}
& \leq \frac{2\log{(11\sigma\sqrt{n})}}{\sigma\sqrt{n}} + \sum_{k\in \mathbb{Z}}{\sup_{x \in [k,k+1)}|g_n(k, x)|}\Bigl(\log{(11\sigma\sqrt{n})} + \log(e2^n\sigma\sqrt{n})\Bigr) \\
&\leq \frac{2\log{(11\sigma\sqrt{n})}}{\sigma\sqrt{n}} + \frac{2^{n+1}}{\sigma\sqrt{n}}\log(11e2^n\sigma\sqrt{n}) \\
&\leq  \frac{3}{4}\frac{2^{n+2}}{\sigma\sqrt{n}}\log(11e2^{n}\sigma\sqrt{n}) \\ \label{lastmucorr}
&\leq  \frac{2^{n+2}}{\sigma\sqrt{n}}\log((11e)^{3/4}2^{n}\sigma^{3/4}\sqrt{n}) \leq \frac{2^{n+2}}{\sigma\sqrt{n}}\log(2^{n+2}\sigma\sqrt{n}),
\end{align}
where $g_n(k,x)$ is given by Lemma \ref{densityapproxlemma} applied to the log-concave random variable $S_n$ and therefore 
$\sum_k{\sup_{x \in [k,k+1)}g_n(k, x)} \leq \frac{2^n}{\sigma\sqrt{n}}.$ In the last inequality in \eqref{lastmucorr} we have used that $\frac{(11e)^{3/4}}{\sigma^{1/4}} \leq 4,$ for $\sigma > 3^7$.
Therefore, by \eqref{minusaddns} and \eqref{lastmucorr},
\begin{align} \nonumber
&\Bigl| H(S_n) - \sum_{k \in (-5\ns,5\ns)}{\int_{[k,k+1)}{F(f_{S_n+\Un}(x)) dx}} \Bigr| \\ 
\nonumber
&\leq \Bigl| H(S_n) -\sum_{k \in  (-5\ns,5\ns)}{G(p_{S_n}(k))} -  \log{\ns}\mathbb{P}\bigl(S_n+\Un \in  (-\lfloor 5\ns \rfloor,\lfloor 5\ns \rfloor +1)\bigr) \Bigr| 
\\
&+ \frac{2^{n+2}}{\sigma\sqrt{n}}\log(2^{n+2}\sigma\sqrt{n}) 
\end{align}
\begin{align}
\nonumber
&\leq   \Bigl| H(S_n) -\sum_{k \in  (-5\ns,5\ns)}{F(p_{S_n}(k))}\Bigr| + \frac{2^{n+2}}{\sigma\sqrt{n}}\log(2^{n+2}\sigma\sqrt{n})   \\
& + \log{\ns}\Bigl| \mathbb{P}\bigl(S_n+\Un \in  (-\lfloor 5\ns \rfloor,\lfloor 5\ns \rfloor+1)\bigr) - \mathbb{P}\bigl(S_n\in  (-5\ns,5\ns+ 1)\bigr)\Bigr|\\ \nonumber
&\leq \sum_{|k| \geq 5\ns}{F(p_{S_n}(k))} + \frac{2^{n+2}}{\sigma\sqrt{n}}\log(2^{n+2}\sigma\sqrt{n}) \\ \label{lastHSnabs}
&+ \log{\ns}\Bigl| \mathbb{P}\bigl(S_n+\Un \in  (-\lfloor 5\ns \rfloor,\lfloor 5\ns \rfloor+1)\bigr) - \mathbb{P}\bigl(S_n\in  (-5\ns,5\ns+ 1)\bigr)\Bigr|.
\end{align}
But, in view of \eqref{ctsdiscretetails}, we can bound the discrete tails in the same way: 
\begin{equation} \label{discretetailsbound}
\sum_{|k| \geq 5\ns}{F(p_{S_n}(k))} \leq 2^{n+5} e^{-(\sqrt{n}\sigma)^{1/5}} (\sqrt{n}\sigma)^{3}.
\end{equation}
\noindent
Finally, note that by Chebyshev's inequality
\begin{align} \label{probscheb}
0\leq \mathbb{P}\bigl(S_n+\Un \notin  (-\lfloor5\ns\rfloor,\lfloor5\ns\rfloor + 1)\bigr) &\leq \mathbb{P}\bigl(|S_n+\Un - \frac{n}{2}| > 4\ns\bigr) \leq \frac{1}{8\ns}
\end{align}
and the same upper bound applies to $\mathbb{P}\bigl(S_n\notin  (-5\ns,5\ns+ 1)\bigr)$. Since both probabilities inside the absolute value in \eqref{lastHSnabs} are also upper bounded by $1$,
replacing the bounds \eqref{discretetailsbound} and \eqref{probscheb} into \eqref{lastHSnabs}, we get 
\begin{align} \nonumber
&\Bigl| H(S_n) - \sum_{k \in (-5\ns,5\ns)}{\int_{[k,k+1)}{F(f_{S_n+\Un}(x)) dx}} \Bigr| \\ 
&\leq  2^{n+5} e^{-(\sqrt{n}\sigma)^{1/5}} (\sqrt{n}\sigma)^{3} + \frac{2^{n+2}}{\sigma\sqrt{n}}\log(2^{n+2}\sigma\sqrt{n}) +\frac{\log{\ns}}{8\ns}.
\end{align}

In view of \eqref{sumbreak} and the bounds on the continuous tails \eqref{tailfinalbound},\eqref{lefttail} we conclude 
\begin{align}
\bigl|h(S_n+\Un) - H(S_n)\bigr| \leq 2^{n+6} e^{-(\sqrt{n}\sigma)^{1/5}} (\sqrt{n}\sigma)^{3} + \frac{2^{n+2}}{\sigma\sqrt{n}}\log(2^{n+2}\sigma\sqrt{n}) +\frac{\log{\ns}}{8\ns}
\end{align}
as long as \eqref{pmaxsmall} and \eqref{sigma3tothe7} are satisfied, that is as long as $\sigma > \max\{2^{n+2}/\sqrt{n},3^7/\sqrt{n}\}$.

\end{proof}

\begin{remark}
The exponent in \eqref{tailfinalbound} can be improved due to the suboptimal step \eqref{suboptimal} in Lemma \ref{nzerolemma}. However, this is only a third-order term and therefore the rate in Theorem \ref{maintheorem} would still be of the same order. 
\end{remark}

%
%

\begin{proof}[Proof of Theorem \ref{taosconjecturelog}.]
Let $U_1,\ldots,U_n$ be continuous i.i.d. uniforms on $(0,1).$ Then by the generalised entropy power inequality for continuous random variables \cite{artstein, mokshayepi}
\begin{equation}
h\Bigl(\frac{X_1+\cdots+X_{n+1}+U_1+\cdots+U_{n+1}}{\sqrt{n+1}}\Bigr) \geq h\Bigl(\frac{X_1+\cdots+X_{n}+U_1+\cdots+U_{n}}{\sqrt{n}}\Bigr).
\end{equation}
But by the scaling property of differential entropy \cite{cover} this is equivalent to 
\begin{equation} \label{finalscaling}
 h(X_1+\cdots+X_{n+1} + U_1+\cdots+U_{n+1}) \geq  h(X_1+\cdots+X_{n} + U_1+\cdots+U_{n}) + \frac{1}{2}\log{\Bigl(\frac{n+1}{n}\Bigr)}.
\end{equation}
Now we claim that for every $n\geq 1$, if $H(X_1) \geq \log{\frac{2}{\epsilon}} + \log{\log{\frac{2}{\epsilon}}} + n + 26$ then
\begin{equation} \label{endofproof}
| h(X_1+\cdots+X_{n} + U_1+\cdots+U_{n}) - H(X_1+\cdots+X_{n})| \leq \frac{\epsilon}{2}.
\end{equation} 
Then the result follows from \eqref{endofproof}, applied to both sides of \eqref{finalscaling} (for $n$ and $n+1$ respectively). 

To prove the claim \eqref{endofproof} we invoke Theorem \ref{maintheorem}. To this end let $n \geq 1$ and assume that $H(X_1) \geq \log{\frac{2}{\epsilon}} + \log{\log{\frac{2}{\epsilon}}} + n + 26$ . First we note, that since \cite{cover}
\begin{equation} \label{rate2}
H(X_1) = h(X_1 + U_1) \leq \frac{1}{2}\log{(2\pi e(\sigma^2 +1/12))},
\end{equation}
we have $e^{H(X_1)} \leq 6\sigma$ provided that $\sigma > 0.275$. Thus, $H(X_1) \geq 26$ implies $\sigma > 50^6 > 90^5$. Therefore the assumptions of the theorem are satisfied and we get 
\begin{align} \nonumber
&| h(X_1+\cdots+X_{n} + U_1+\cdots+U_{n}) - H(X_1+\cdots+X_{n})| \\
&\leq 2^{n+6} e^{-(\sqrt{n}\sigma)^{1/5}} (\sqrt{n}\sigma)^{3} + \frac{2^{n+2}}{\sigma\sqrt{n}}\log(2^{n+2}\sigma\sqrt{n}) +\frac{\log{\ns}}{8\ns} \\ \label{numericaljusts}
&\leq \bigl(2^6+2^3+1\bigr)2^{n}\frac{\log{\sqrt{n}\sigma}}{\sqrt{n}\sigma} 
= 73\cdot2^n\frac{\log{\sqrt{n}\sigma}}{\sqrt{n}\sigma}.
\end{align}
In \eqref{numericaljusts} we used the elementary fact that for $x \geq 90^5$, $\frac{x^3}{e^{x^{1/5}}} \leq \frac{1}{x} \leq \frac{\log{x}}{x}$ to bound the first term and the assumption 
$\sigma\sqrt{n} \geq 2^{n+2}$ to bound the second term. 
Thus, by assumption $\sigma \geq \frac{e^{H(X_1)}}{6} \geq \frac{2}{\epsilon}\log{\frac{2}{\epsilon}}e^{n+24}$ and since $ \frac{\log{x}}{x}$ is non-increasing for $x > {e}$, we obtain by \eqref{numericaljusts}
\begin{align} 
&| h(X_1+\cdots+X_{n} + U_1+\cdots+U_{n}) - H(X_1+\cdots+X_{n})|  \\
&\leq 73\cdot 2^n\frac{\log{\frac{2}{\epsilon}} + \log{\log{\frac{2}{\epsilon}}} + n + 24}{\frac{2}{\epsilon}\log{\frac{2}{\epsilon}e^ne^{24}}} \\
&\leq \frac{\epsilon}{2}\Bigl[\frac{146}{(\frac{e}{2})^ne^{24}}  + \frac{n73}{\log{\frac{2}{\epsilon}}(\frac{e}{2})^ne^{24}} + \frac{1752}{\log{\frac{2}{\epsilon}}(\frac{e}{2})^ne^{24}}\Bigr] \\
&< \frac{\epsilon}{2}
\end{align}
proving the claim \eqref{endofproof} and thus the theorem. 

\end{proof}

\appendixpage

\begin{appendices}

\section{An elementary Lemma}

Here we prove the following Taylor-type estimate that we used in the proof of Theorem \ref{maintheorem}. A similar estimate was used in \cite{tao_sumset_entropy}.
\begin{lemma} \label{elementaryestimatelemma}
Let $D, M \geq 1$ and, for $x > 0,$ consider $G(x) = F(x) -x\log{M}$, where $F(x) = -x\log{x}.$
Then, for $0 \leq a,b \leq \frac{D}{M}$ and any $0 < \mu < \frac{1}{e},$
we have the estimate
\begin{equation} 
|G(b) - G(a)| \leq  \frac{2\mu}{M}\log{\frac{1}{\mu}} + |b-a|\bigl[\log{\frac{1}{\mu}} + \log{(eD)}\bigr].
\end{equation}

\end{lemma}
\begin{proof}

Note that $\p{G}(x) = -\log{x} - 1 - \log{M}$ , which is non-negative for $x < \frac{1}{eM}$.\\
We will consider two cases separately. 

The first case is when either $a < \frac{\mu}{M}$ or $b < \frac{\mu}{M}$. Assume without loss of generality that $a < \frac{\mu}{M}$. Then if $b < \frac{\mu}{M}$ as well, we have $|G(b) - G(a)| \leq G(a) + G(b) \leq \frac{2\mu}{M}\log{\frac{1}{\mu}},$ since then $\p{G} \geq 0$. 
On the other hand, if $b \geq \frac{\mu}{M}> a$ then $G(a) \geq a\log{\frac{1}{\mu}}$ and $G(b) \leq b\log{\frac{1}{\mu}}$. \\
But then, either $G(b) > G(a),$ whence $|G(b) - G(a)| \leq |b-a|\log{\frac{1}{\mu}}$ or $G(b) < G(a),$ whence $|G(b)-G(a)| \leq |b-a|\log{(eD)},$ since we must have $G(a) - G(b) = (a-b)G^{\prime}(\xi),$ for some $\xi \in (\frac{1}{eM},\frac{D}{M}]$. \\
Thus, in the first case, $|G(b) - G(a)| \leq \frac{2\mu}{M}\log{\frac{1}{\mu}} + |b-a|\bigl[\log{\frac{1}{\mu}} + \log(eD)\bigr]$.

The second case, is when both $a,b \geq  \frac{\mu}{M}$. Then $G(b) - G(a) = (b-a)\p{G}(\xi)$, for some $\frac{D}{M} \geq \xi \geq \mu\frac{1}{M}$.\\
Since then $|\p{G}(\xi)| \leq \log{\frac{1}{\mu}} + \log{D}+1$,
we have $|G(b) - G(a)| \leq |b-a|(\log{\frac{1}{\mu}} + \log{(eD)})$.

In any case, $|G(b) - G(a)| \leq  \frac{2\mu}{M}\log{\frac{1}{\mu}} + |b-a|\bigl[\log{\frac{1}{\mu}} + \log{(eD)}\bigr].$
\end{proof}
\end{appendices}
%
%

\begin{thebibliography}{10}
\bibitem{artstein}
S.~Artstein, K.~Ball, F.~Barthe, and A.~Naor, ``Solution of {S}hannon’s
  problem on the monotonicity of entropy,'' \emph{Journal of the American
  Mathematical Society}, vol.~17, no.~4, pp. 975--982, 2004.

\bibitem{barronentropy}
A.~R. Barron, ``Entropy and the central limit theorem,'' \emph{The Annals of
  probability}, pp. 336--342, 1986.

\bibitem{bobkov2022concentration}
S.~G. Bobkov, A.~Marsiglietti, and J.~Melbourne, ``Concentration functions and
  entropy bounds for discrete log-concave distributions,'' \emph{Combinatorics,
  Probability and Computing}, vol.~31, no.~1, pp. 54--72, 2022.

\bibitem{cover}
T.~M. Cover and J.~A. Thomas, \emph{Elements of Information Theory (Wiley
  Series in Telecommunications and Signal Processing)}.\hskip 1em plus 0.5em
  minus 0.4em\relax New York, NY, USA: Wiley-Interscience, 2006.

\bibitem{daviselementary}
B.~Davis and D.~McDonald, ``An elementary proof of the local central limit
  theorem,'' \emph{Journal of Theoretical Probability}, vol.~8, no.~3, pp.
  693--701, 1995.

\bibitem{gavalakis2021entropy}
L.~Gavalakis and I.~Kontoyiannis, ``Entropy and the discrete central limit
  theorem,'' \emph{arXiv preprint arXiv:2106.00514}, 2021.

\bibitem{adaptivesensing}
S.~Haghighatshoar, E.~Abbe, and E.~Telatar, ``Adaptive sensing using
  deterministic partial {H}adamard matrices,'' in \emph{2012 IEEE International
  Symposium on Information Theory Proceedings}.\hskip 1em plus 0.5em minus
  0.4em\relax Ieee, 2012, pp. 1842--1846.

\bibitem{abbe}
S.~Haghighatshoar, E.~Abbe, and I.~E. Telatar, ``A new entropy power inequality
  for integer-valued random variables,'' \emph{IEEE Transactions on Information
  Theory}, vol.~60, no.~7, pp. 3787--3796, 2014.

\bibitem{harremoesepibinomial}
P.~Harremo{\'e}s, C.~Vignat, ``An entropy power inequality for
  the binomial family,'' \emph{JIPAM. J. Inequal. Pure Appl. Math}, vol.~4,
  no.~5, p.~93, 2003.

\bibitem{hoggar}
S.~Hoggar, ``Chromatic polynomials and logarithmic concavity,'' \emph{Journal
  of Combinatorial Theory, Series B}, vol.~16, no.~3, pp. 248--254, 1974.

\bibitem{mokshayepi}
M.~Madiman and A.~Barron, ``Generalized entropy power inequalities and
  monotonicity properties of information,'' \emph{IEEE Transactions on
  Information Theory}, vol.~53, no.~7, pp. 2317--2329, 2007.

\bibitem{mcdonald}
D.~R. McDonald, ``On local limit theorem for integer-valued random variables,''
  \emph{Theory of Probability \& Its Applications}, vol.~24, no.~3, pp.
  613--619, 1980.


\bibitem{mineka}
J.~Mineka, ``A criterion for tail events for sums of independent random
  variables,'' \emph{Zeitschrift f{\"u}r Wahrscheinlichkeitstheorie und
  Verwandte Gebiete}, vol.~25, no.~3, pp. 163--170, 1973.

\bibitem{ruzsa2009sumsets}
I.~Z. Ruzsa, ``Sumsets and entropy,'' \emph{Random Structures \& Algorithms},
  vol.~34, no.~1, pp. 1--10, 2009.

\bibitem{shannon1948mathematical}
C.~E. Shannon, ``A mathematical theory of communication,'' \emph{The Bell
  system technical journal}, vol.~27, no.~3, pp. 379--423, 1948.

\bibitem{stam}
A.~J. Stam, ``Some inequalities satisfied by the quantities of information of
  {Fisher and Shannon},'' \emph{Information and Control}, vol.~2, no.~2, pp.
  101--112, 1959.

\bibitem{tao_sumset_entropy}
T.~Tao, ``Sumset and inverse sumset theory for {S}hannon entropy,''
  \emph{Combinatorics, Probability and Computing}, vol.~19, pp. 603--639, 07
  2010.

\bibitem{taovunotes}
T.~Tao and V.~H. Vu, ``Entropy methods,'' 2005. [Online]. Available:
  \href{http://www.math.ucla.edu/~tao/preprints/Expository/}{http://www.math.ucla.edu/~tao/preprints/Expository/}

\bibitem{tao_vu_book}
------, \emph{Additive Combinatorics}, ser. Cambridge Studies in Advanced
  Mathematics.\hskip 1em plus 0.5em minus 0.4em\relax Cambridge University
  Press, 2006.

\bibitem{woo2015discrete}
J.~O. Woo and M.~Madiman, ``A discrete entropy power inequality for uniform
  distributions,'' in \emph{2015 IEEE International Symposium on Information
  Theory (ISIT)}.\hskip 1em plus 0.5em minus 0.4em\relax IEEE, 2015, pp.
  1625--1629.

\end{thebibliography}
%

\section*{Acknowledgements}

The author is indebted to Ioannis Kontoyiannis for interesting discussions as well as many useful suggestions and comments. The author would also like 
to thank the two anonymous reviewers for the careful reading of the manuscript and for many useful comments that significantly improved the presentation of the results in the paper.

\end{document}